\documentclass[HARVARD,LATO2COL]{WileyNJDv5}

\articletype{Article Type}%

\received{Date Month Year}
\revised{Date Month Year}
\accepted{Date Month Year}
\journal{Journal}
\volume{00}
\copyyear{2023}
\startpage{1}

\usepackage{graphicx}
\usepackage{amsmath,amsthm,amsfonts,amssymb}
\usepackage{color}
\usepackage{bm}
\usepackage{bigints}
\usepackage{subcaption}
\usepackage{float}

\usepackage{color}

\newcommand{\e}{\mathbb{E}}
\newcommand{\p}{\mathbb{P}}
\newcommand{\lk}{\left[ }
\newcommand{\rk}{\right] }
\newcommand{\lc}{\left(}
\newcommand{\rc}{\right)}
\newcommand{\R}{\mathbb{R}}

\newcommand{\N}{\mathbb{N}}

\newcommand{\sumn}{\sum_{i=1}^n}

\newcommand{\rmd}{\mathrm{d}}
\newcommand{\sumijn}{\sum_{\substack{i,j=1 \\ i\not=j} }^n}

\newcommand{\lv}{\left|}
\newcommand{\rv}{\right|}

\newtheorem{Definition}{Definition}
\newtheorem{Theorem}{Theorem}
\newtheorem{Lemma}{Lemma}
\newtheorem{prop}{Proposition}
\newtheorem{cor}{Corollary}
\newtheorem{Remark}{Remark}

\newcommand*{\RR}{\ensuremath{\mathbb{R}}}	        	
\newcommand*{\QQ}{\ensuremath{\mathbb{Q}}}	        	
\newcommand*{\NN}{\ensuremath{\mathbb{N}}}		        
\newcommand*{\EE}{\ensuremath{\mathbb{E}}}	        	
\newcommand*{\PP}{\ensuremath{\mathbb{P}}}	        	
\newcommand*{\Var}{\ensuremath{\text{Var}}}	        	

\newcommand{\thetanull}{\ensuremath{{\theta_0} } }

\newcommand{\thetanhat}{\ensuremath{ {\hat{\theta}_{n}} } }

\newcommand*{\XX}{\ensuremath{\mathcal{X}}}

\newcommand{\tn}{\thetanhat}
\newcommand{\too}{\theta_0}
\newcommand{\tns}{\theta_n^\star}

\DeclareMathOperator*{\argmin}{arg\,min}

\DeclareMathOperator{\ksd}{KSD}
\DeclareMathOperator{\hksd}{\widehat{\ksd}}

\algrenewcommand\algorithmicensure{\textbf{Decision:}} 

\begin{document}

\title{Composite goodness-of-fit test with the Kernel Stein Discrepancy and a bootstrap for degenerate $U$-statistics with estimated parameters}

\author[1]{Fabian Baier}

\author[2]{Florian Brück}

\author[1]{Veronika Reimoser}


\address[1]{\orgdiv{Chair of Mathematical Finance}, \orgname{Technical University of Munich}, \orgaddress{\country{Germany}}}

\address[2]{\orgdiv{Research Institute for Statistics and Information Science}, \orgname{University of Geneva}, \orgaddress{ \country{Switzerland}}}

\corres{Corresponding author Florian Brück \email{florian.brueck.edu@gmail.com}}



\abstract[Abstract]{    This paper formally derives the asymptotic distribution of a goodness-of-fit test based on the Kernel Stein Discrepancy introduced in (Oscar Key et al. ``Composite Goodness-of-fit Tests with Kernels'', Journal of Machine Learning Research 26.51 (2025), pp. 1–60). The test enables the simultaneous estimation of the optimal parameter within a parametric family of candidate models. Its asymptotic distribution is shown to be a weighted sum of infinitely many $\chi^2$-distributed random variables plus an additional disturbance term, which is due to the parameter estimation. Further, we provide a general framework to bootstrap degenerate parameter-dependent $U$-statistics and use it to derive a new Kernel Stein Discrepancy composite goodness-of-fit test.}



\maketitle

\renewcommand\thefootnote{}

\renewcommand\thefootnote{\fnsymbol{footnote}}
\setcounter{footnote}{1}

\section{Introduction}
\label{sec:intro}
Assessing the goodness-of-fit of statistical models to observed data is crucial in both machine learning and statistics. More precisely, for a considered distribution $P$ and observed data $\{X_i\}_{i\in[n]}$ from another distribution $Q$, goodness-of-fit tests compare the null hypothesis $H_0: \ P=Q$ against the alternative hypothesis $H_1: \ P\neq Q$.
Traditional tests such as the Kolmogorov-Smirnov, Cramér-von-Mises and Anderson-Darling test compare cumulative distribution functions. However, they struggle with high dimensional data and models where the cumulative distribution functions or the exact likelihoods are intractable. 

In case the likelihood is known up to normalizing constant, the so-called Kernel Stein Discrepancy (KSD) can be used for goodness-of-fit testing, see \cite{Gorham_2015,Liu_2016,Chwialkowski_2016} and \cite{steinreview2023} for a comprehensive overview of its uses in statistics and machine learning. Under some regularity conditions, $\ksd(P,Q)=0$ iff $P=Q$. A hypothesis test of $H_0$ is then based on an estimate of $\ksd(P,Q)$, rejecting $H_0$ when $\ksd(P,Q)$ is sufficiently large. 
Going beyond testing whether $P=Q$, we can also ask whether the data originates from any member of a parametric family $\{P_\theta\}_{\theta\in\Theta}$, leading to composite goodness-of-fit tests. Their null hypothesis and alternative are  
$$H_0^C: \exists \ \thetanull\in\Theta: Q=P_\thetanull \quad \text{against the alternative} \quad H_1^C: Q\notin \{P_\theta\}_{\theta\in\Theta}.$$
In \cite{Key_2023}, two composite goodness-of-fit tests were proposed, one based on the KSD and one based on the so-called Maximum Mean Discrepancy (MMD). The authors have rigorously derived a composite goodness-of-fit testing framework for the MMD, allowing $\theta_0$ to be estimated simultaneously. Further, they suggested a similar framework for the KSD without rigorously proving its validity stating ``\textit{We also include encouraging empirical results for $KSD$ but leave the extension of our theoretical framework to this test for future work}''. In this paper, we fill this gap in the literature and formally derive the asymptotic distribution of the KSD estimator proposed in \cite{Key_2023}, under the null and the alternative.

Formally, the KSD is defined for two probability measures $P$ and $Q$ which have Lebesgue densities $p$ and $q$, respectively. Omitting some regularity conditions, the KSD is given by
$$ \ksd_q(p):=\sqrt{\mathbb{E}_{X,X' \sim q}[h_p(X, X')]}, $$
where $h_p(X,X^\prime)$ is a function that depends only on $p$ in terms of $\nabla\log(p)$. It is now easy to see that $\ksd^2_q(p)$ can be estimated by the corresponding $U$-statistic of order two given by
\begin{align}
    \hksd^2_q(p)=\frac{1}{n(n-1)}\sum_{i,j\in[n],i\not=j} h_p(X_i,X_j), \label{eqndefksdestimator}
\end{align}
where $(X_i)_{i\in[n]}$ denotes an i.i.d.\ sample from $Q$, referring to \cite{Serfling_1980} for an introduction to $U$-statistics. 
 In particular, the KSD can be estimated using an unnormalized version of $p$, since $\nabla\log(p)$ is independent of the normalizing constant. This property makes the KSD especially useful when the normalizing constant is intractable or unknown.
 For our composite goodness-of-fit testing framework, we assume that $P_\theta(\rmd x)=p_\theta(x)\rmd x$ for all $\theta\in\Theta$ and that we have an estimator $\hat{\theta}_n$ of $\theta_0$ at hand, which allows to estimate $\ksd_q^2(p_{\theta_0})$ via $\hksd^2_q(p_{\hat{\theta}_n})$. However, the main difficulty when working with the KSD is that, under $H_0^C$, $\hksd^2_q(p_{\thetanull})$ is a degenerate $U$-statistic with convergence rate $n^{-1}$, whereas, under $H_1^C$, $\hksd^2_q(p_{\thetanull})$ is typically a non-degenerate $U$-statistic with convergence rate $n^{-1/2}$. This significantly complicates the derivation of the asymptotic distribution of $\hksd^2_q(p_{\thetanhat})$ and even more so the estimation of critical values of the corresponding goodness-of-fit test. It should be noted that \cite{dewet1987} and \cite{leuchtneumann2009} also provided a framework for the derivation of the asymptotic distribution of degenerate $U$-statistics under parameter estimation, however, their frameworks are not applicable in our context as they require that either the $U$-statistics core follows a specific form or that it is degenerate w.r.t.\ the empirical measure of the data generating process, which is usually not satisfied in our setting. 

To propose a theoretically valid hypothesis test, the quantiles of the asymptotic distribution also have to be estimated consistently. For the KSD, \cite{Key_2023} proposed a parametric and a wild bootstrap scheme for composite goodness-of-fit tests based on the KSD, but did not provide any theoretical validation. Here, we propose an alternative framework and prove its theoretical validity, even when we estimate $\theta_0$ through $\hat{\theta}_n$. Further, we show that the wild bootstrap of \cite{Key_2023} does not yield a theoretically valid test, and illustrate that this issue arises more generally when the wild bootstrap is naively extended to degenerate $U$-statistics.
Finally, we want to emphasize that our proposed bootstrap framework is not only valid for the particular case of the KSD. Instead, we provide a general approach to bootstrap degenerate parameter-dependent $U$-statistics, which is a result of independent interest beyond the scope of this paper.

The paper is structured as follows:
We start with an introduction to the KSD. We then provide assumptions under which we derive the asymptotic distribution of the estimator of the KSD while simultaneously estimating a parameter. Then, we provide a general framework to bootstrap parameter-dependent degenerate $U$-statistic and apply it to the case of the KSD. Finally, we illustrate our method in a short simulation study. All proofs are deferred to the Appendix.

\section{The KSD}\label{sec:background}
In the following, we rigorously introduce the KSD. Let $\XX\subseteq \RR^d$ with a non-empty interior and let $q$ be a strictly positive and continuously differentiable density of the probability measure $Q$ on $\RR^d$ with support $\XX$. Further, for any function $f:\RR^d\to\RR$ we denote $\nabla_x f:=\lc\frac{\partial}{\partial x_i}f(x)\rc_{1\leq i\leq d}$ for the gradient of $f$, implicitly assuming existence, and denote $s_q:=(s_{q,i})_{1\leq i\leq d}= \nabla_x\log(q)$ for the score function.

\begin{Definition}
We say that a function $f : \XX \rightarrow \mathbb{R}$ is in the Stein class of $q$ if $f$ is continuously differentiable and satisfies $\int_\XX \nabla_x(f(x)q(x)) \,dx = \bm{0}$. A positive definite kernel $k:\XX\times \XX\to \RR$ is said to be in the Stein class of $q$ if $k$ has continuous second-order partial derivatives, and $k(x, \cdot)$ is in the Stein class of $q$ for any fixed $x\in\XX$. Further, we say that $k$ is integrally strictly positive definite if for any function $g$ s.t.\ $0<\int_\XX g(x)^2\mathrm d x <\infty$ we have that $\int_\XX\int_\XX g(x)k(x,x^\prime)g(x^\prime)\rmd x\rmd x^\prime>0$.
\end{Definition}

For example, it is easy to check that the Gaussian kernel $k(x,x^\prime)=\exp\lc -\Vert x-x^\prime\Vert_2^2 /(2l^2)\rc$ is integrally strictly positive definite and in the Stein class of any continuously differentiable density with support $\RR^d$.

In the following, we work with the definition of the KSD from \cite[Definition 3.2, Proposition 3.3, Theorem 3.6]{Liu_2016}, as this representation is suitable for estimation.

\begin{Definition}
Let $k$ be in the Stein class of $q$, $p$ denote a strictly positive and continuously differentiable density with support $\XX$, $\bm{\delta}_{p,q}(X) = \bm{s}_p\, (X) - \bm{s}_q\, (X)$ denote the score difference between $p$ and $q$ and assume that $\int_\XX \lc q(x)^T \bm{\delta}_{p,q}(x)\rc^2 \mathrm{d} x<\infty$. Suppose $p$, $q$, and $k$ are such that $\mathbb{E}_{X,X' \sim q}[\bm{\delta}_{p,q}(X)^T k(X, X') \bm{\delta}_{p,q}(X')]<\infty$. Then, we define the KSD as
\begin{align} 
    \ksd^2_q(p) = \mathbb{E}_{X,X' \sim q}[h_p(X, X')] \label{eq.: KSD for statistic}
\end{align}
where $X,X^\prime$ are independent and $h_p:\XX\times\XX \to \RR$ is defined as
\begin{align}
  h_p(x, x') &= \bm{s}_p(x)^T k(x, x') \bm{s}_p(x') + \bm{s}_p(x)^T \nabla_{x'} k(x, x') + \nabla_x k(x, x')^T \bm{s}_p(x') \nonumber\\
  &  \ \ + \sum_{i=1}^d \frac{\partial^2 }{\partial x_i \partial x'_i}k(x,x') \label{kernelksd}
\end{align}
\end{Definition}

This definition ensures that the KSD is well-defined and that $\ksd_q(p)=0\Leftrightarrow q=p$.
The finiteness of $\int_\XX \lc q(x)^T \bm{\delta}_{p,q}(x)\rc^2 \mathrm{d} x<\infty$ heavily depends on the tails of $p$ and $q$ and needs to be checked on a case-by-case basis. It is now straightforward to see that (\ref{eqndefksdestimator}) is an estimator of $\ksd^2_q(p)$ and that, whenever $\thetanhat$ is an estimator of $\thetanull$, $\hksd^2_q(p_{\thetanhat})$ is an estimator of $\ksd^2_q(p_{\theta_0})$.

\section{Asymptotic distribution of the composite KSD estimator}
\label{sec:asymptotic_distr}

\subsection{Assumptions}

We now state assumptions that ensure the KSD is always well-defined and enable us to derive the asymptotic distribution of $\hksd_q^2(p_\thetanhat)$. As a high-level assumption that is employed throughout the paper we remark that all random variables appearing in the following are assumed to originate from an abstract probability space $(\Omega,\mathcal{F},\PP)$. Therefore, statements such $o_\PP(1)$ and $O_\PP(1)$ always refer to convergence to $0$ in probability and boundedness in probability on this abstract probability space. Moreover, the expectation of a vector valued function $f=(f_1,\ldots,f_d)$ is interpreted componentwise, i.e.\ $\lc \EE\lk f_1 \rk,\ldots, \EE\lk f_d \rk\rc$ and similarly for matrices. Moreover, for any norm $\Vert \cdot\Vert$ on $\R$ and any matrix $A=\lc A_{i,j}\rc_{1\leq i,j\leq d^\prime}$ we define $\Vert A\Vert:=\lc \Vert A_{i,j} \Vert \rc_{1\leq i,j\leq d^\prime}$. Finally, $\nabla_x^m f$ denotes the collection of all $m$-th partial derivatives of a function $f$ w.r.t.\ $x$, i.e.\ $\nabla_x f$ is the gradient and $\nabla_x^2 f$ is the Hesse matrix of $f$, where we use the convention $\nabla^0 f:=f$.

\begin{assumption}\label{KSD assumpt. 0}
    The observations $(X_i)_{i\in\NN}$ are i.i.d.\ with distribution $Q$ which has a continuously differentiable density $q$ on $\XX\subset \RR^d$, where $\XX$ has non-empty interior.
\end{assumption}

\begin{assumption}\label{KSD assumpt. 1}
$\Theta$ is a compact and convex subset of $\RR^p$ with non-empty interior.  
$\theta_0\in \argmin_{\theta\in\Theta}\ksd_q(p_\theta)$ belongs to the interior of $\Theta$.    
\end{assumption}

\begin{assumption} \label{KSD assumpt. 5}
The kernel $k: \XX \times \XX \to \RR$ is bounded, strictly integrally positive definite, in the Stein class of $q$ and has bounded first order derivatives.
\end{assumption}
Not every kernel satisfies this assumption, but the Gaussian kernel, for example, has bounded first and second order derivatives.

\begin{assumption}\label{KSD assumpt. 3}
$\{P_\theta\}_{\theta \in \Theta}$ is an identifiable parametric family of models on $\XX$. Each element $P_\theta\in \{P_\theta\}_{\theta \in \Theta}$ has continuously differentiable density $p_\theta$ w.r.t. the Lebesgue measure and support $\XX$ and satisfies:
\begin{enumerate}
    \item For every $\theta\in\Theta$ the map $x\mapsto p_\theta(x)$ is strictly positive and continuously differentiable on $\XX$.
    \item For every $x \in \XX$ we have $\theta \mapsto p_\theta (x) \in \mathcal{C}^4(\Theta)$.
    \item  $\e\lk \sup_{\theta\in\Theta}  \Vert \nabla^m s_{p_\theta}(X)\Vert_1^4\rk<\infty$ for all $m\in\{0,1\}$.
    \item $\e\lk \sup_{\theta\in\Theta}\Vert  \nabla^m s_{p_\theta}(X) \Vert_1^2\rk<\infty$  for all $m\in\{2,3\}$.
    \item $\int_\XX  q(x)^T \lc s_{p_\theta}(x) - s_q(x)\rc^2\rmd x<\infty$ for all $\theta\in\Theta$.
\end{enumerate}
\end{assumption}
Essentially, this assumption ensures that $\theta_0$ is unique and that $h_{\theta}$ has high-enough moments such that the corresponding KSD is well-defined and its estimator $\hksd$ converges asymptotically to a non-degenerate distribution.

The next assumption requires the existence of the joint asymptotic distribution of a $2p+1$-dimensional random vector, denoting $\nabla_\theta\hksd^2_q(p_\theta) :=\frac{1}{n (n-1)} \sum_{i,j\in[n], i\neq j} \nabla_{\theta} h_\theta(X_i,X_j)$ for all $\theta\in\Theta$, where we use the short-hand notation $h_\theta:=h_{p_\theta}$.  
\begin{assumption}\label{KSD assumpt. 6}
Under $H_0^C$, the random vector
$\Big(( n\hksd^2_q(p_\thetanull) ,\ \sqrt{n} \nabla_\theta\hksd^2_q(p_{\theta_0}) ,
\sqrt{n}(\thetanhat-\thetanull) \Big)$
converges weakly to a real-valued random vector $(Z_1,Z_2,Z_3)$ when $n\to\infty$.
\end{assumption}
Note that, under $H_0^C$ and Assumptions \ref{KSD assumpt. 0}, \ref{KSD assumpt. 5} and \ref{KSD assumpt. 3}, Corollary \ref{cor:asymptist} in the appendix shows marginal converge of $\left( n\hksd^2_q(p_\thetanull), \sqrt{n}\nabla_\theta\hksd^2_q(p_\theta)\right)\to (Z_1,Z_2)$, where $Z_1\sim \sum_{j=1}^{\infty} \lambda_j (T_j^2 - 1)$ with $(T_j)_{j\in\NN}\overset{i.i.d.} {\sim} \mathcal{N}(0,1)$ and the $\lambda_j$ are the eigenvalues of the operator $ \psi\mapsto \EE_{X\sim Q}[h_\thetanull(x,X) \psi(X)]$, and $Z_2\sim \mathcal{N} \left( 0,4\Var\lc \e_X\lk \nabla_{\theta} h_{\theta_0}(X,X^\prime) \rk\rc \right)$. Further, the joint convergence of $\left( n\hksd^2_q(p_\thetanull) , \sqrt{n}\nabla_\theta\hksd^2_q(p_\theta)\right)$ follows from standard $U$-statistics theory, exploiting the classical orthogonal eigenfunction expansion of the core $h_{\theta_0}$. Thus, Assumption \ref{KSD assumpt. 6} is solely an assumption on the estimator $\thetanhat$. Under some technical regularity conditions, the asymptotic normality of $\thetanhat=\argmin_{\theta\in \Theta} \hksd^2_q(p_\theta)$ has been obtained in \cite[Section 3]{Barp_2019}. As this estimator is closely-related to classical $M$-estimators the assumption of joint normality seems a very weak requirement. However, to keep the framework as general as possible, we want to remark that $\thetanhat$ can be any estimator of $\theta_0$ that obeys the joint convergence assumption. A sufficient condition for Assumption \ref{KSD assumpt. 6} to hold is a Bahadur-type representation of $\thetanhat$ of the form $\sqrt{n} ( \thetanhat -\theta_0 )= n^{-1/2}\sumn \phi_{\theta_0}(X_i) +o_{\p}(1) $ where $\phi_{\theta_0}:\ \R^{d}\to\R^p$ is such that $\e\lk\phi_{\theta_0}(X)^2\rk<\infty$, which directly follows from the joint weak convergence of $U$-statistics and empirical averages. An analogous statement holds when $\sqrt{n}( \thetanhat -\theta_0 )$ can be represented as a scaled $U$-statistic.

\subsection{Asymptotic distribution of $n\hksd_q^2(p_\thetanhat)$}
\begin{Theorem}[Convergence under Null Hypothesis] \label{thm.: Convergence of KSD under H0}
Under Assumptions \ref{KSD assumpt. 0}-\ref{KSD assumpt. 6} and under $H_0^C$, we have
$$n \hksd_q^2(p_\thetanhat) 
\xrightarrow{d} Z_1+Z_2^T Z_3+Z_3^T H^* Z_3=:Z.$$
where $H^*:= \lc\frac{\partial^2}{\partial \theta_i \partial \theta_j} \ksd^2(P_\thetanull, Q) \rc_{i,j\in[p]}$.
\end{Theorem}

The non-composite version of the preceding theorem has already been derived in Theorem 4.1 of \cite{Liu_2016}, where it is shown that $\hksd_q^2(p_\thetanull)\to Z_1$. Comparing the results, we see that the asymptotic distributions differ by the term $Z_2^T Z_3+Z_3^T H^* Z_3$, which corresponds to the additional ``noise'' from parameter estimation.

\begin{Theorem}[Consistency under Alternative Hypothesis]\label{thm.: KSD Consistency under H_1^C}
Under $H_1^C$ and Assumptions \ref{KSD assumpt. 0}-\ref{KSD assumpt. 3}, we have that almost surely
$$
\liminf_{n \to \infty} \hksd^2_q(p_\thetanhat) > 0.
$$
\end{Theorem}
This theorem implies that the test statistic $n \hksd_q^2 (p_\thetanhat)$ diverges under $H_1^C$ as $n \to \infty$, allowing us to conduct a valid hypothesis test whenever we have access to the quantiles of $Z$ under $H_0^C$, which will be the content of the next section.

\section{A bootstrap CLT for degenerate $U$-statistics with estimated parameters}
It is a non-trivial problem to bootstrap a degenerate $U$-statistic, since the naive bootstrap may fail to provide the correct limiting law, as was first shown in \cite{bretagnolle}. Therefore, bootstrapping degenerate $U$-statistics requires extra care and several attempts have been made to solve this problem, see  for example \cite{arconesgine1992}, who provided one of the first frameworks for bootstrapping degenerate $U$-statistics, and \cite{leuchtneumann2013,chwialkowsketali2014} for more recent contributions. In our framework, the problem is even more difficult, as we also need to take into account the effect of the parameter estimation on the limiting distribution of the degenerate $U$-statistic.
Apart from \cite{Key_2023}, the only works we are aware of that address the bootstrapping of degenerate $U$-statistics under concurrent parameter estimation are \cite{dewet1987,leuchtneumann2009}. However, their frameworks require certain properties of the $U$-statistic kernel which are often not verifiable in practice and a particular example of such a $U$-statistic kernel is given by  (\ref{kernelksd}). To obtain a general bootstrap scheme for degenerate, parameter-dependent $U$-statistics this section first shows that a naive extension of the bootstrap for non-degenerate $U$-statistics of \cite{arconesgine1992} does not yield the right asymptotic distribution and we subsequently show how to correctly bootstrap degenerate, parameter-dependent $U$-statistics.

Let us start by introducing some notation. We use $\overset{\star}{\to}$ to denote weak convergence conditional on almost every sequence $(X_i)_{\i\in\NN}$. 
 Moreover, for a sequence of random variables $Z_n$ and a deterministic sequence $a_n$ we write $Z_n=O^\star(a_n)$ (respectively, $Z_n=o^\star(a_n)$) to denote  $Z_n/a_n$ is bounded  (respectively, converges to $0$) in probability, conditional on almost every sequence $(X_i)_{i\in\NN}$.
For an arbitrary function $f$ of two arguments define
$$ U_n f:=\frac{1}{n(n-1)}\sum_{i,j\in[n],i\neq j} f(X_i,X_j) . $$
and define the empirically centered version of $f$ as
$$f_{n}(\cdot,\cdot):=f(\cdot,\cdot) -\EE_{X\sim \QQ_n}\lk f(\cdot, X)\rk-\EE_{X\sim \QQ_n}\lk f(X, \cdot)\rk+\EE_{X,X^\prime\sim \QQ_n}\lk f(X^\prime, X)\rk,$$
where $\QQ_n:=n^{-1}\sum_{1\leq i\leq n}\delta_{X_i}$. It is important to observe that $f_n$ is a degenerate $U$-statistic core w.r.t.\ $\QQ_n$ for every $n\in\N$.

We are ready to describe the bootstrap scheme. We will use the $U$-statistic analogue of Efrons bootstrap introduced in \cite{arconesgine1992}, i.e., we use sampling with replacement from $(X_1,\ldots,X_n)$. We denote $(X_1^\star,\ldots,X_n^\star)$ as a sample of size $n$ from $\mathbb{Q}_n$ and define the bootstrapped version of an arbitrary $U$-statistic with core $f$ as
$$ U_n^\star f:=\frac{1}{n(n-1)}\sumijn f(X^\star_i,X^\star_j). $$

Let us specify our high-level assumptions to derive our bootstrap CLT. The assumptions are formulated for an arbitrary parameter-dependent core $h_\theta$ (not necessarily the KSD core) and ensure implicitly that we already have verified that the non-bootstrapped version of the test statistic $nU_nh_{\thetanhat}$ has a non-degenerate limit. 

\begin{assumption}\label{Ass:bootstrap}
We assume Assumption \ref{KSD assumpt. 0} and the corresponding version of Assumption \ref{KSD assumpt. 6}:
\begin{align*}
    &\lc  n\lc U_n h_{\theta_0}-\e\lk h_{\theta_o}\rk\rc,\sqrt{n}\lc U_n \nabla h_{\theta_0}-\e\lk \nabla h_{\theta_0} \rk\rc,\sqrt{n}(\tn-\too)\rc\\
    &\to (Z_1,Z_2,Z_3)\in \R^{2p+1}.
\end{align*}
Further, we assume
\begin{enumerate}
    \item $\e\lk \sup_{\theta\in\Theta} \Vert \nabla^m  h_\theta(Y_1,Y_2)\Vert_2^2 \rk <\infty$ for $m\in\{0,1\}$ and $\e\lk \sup_{\theta\in\Theta} \Vert \nabla^m  h_\theta(Y_1,Y_2)\Vert_1 \rk <\infty$ for $m\in\{2,3\}$ where $Y_1,Y_2\in \{X,X'\}$ for some i.i.d.\ copy $X'$ of $X$. 
    \item $\Theta$ is compact and convex with non-empty interior and $\theta_0$ is an interior point of $\Theta$. The estimator of $\thetanull$ is a functional of $\QQ_n$, i.e.\ $\thetanhat=\psi(X_1,\ldots,X_n)$.
\end{enumerate}
\end{assumption}

First, we obtain a bootstrap CLT for $\lc nU_n^\star h_{\theta_0,n},\sqrt{n}\lc U_n^\star \nabla h_{\theta_0,n}-U_n \nabla h_{\theta_0}\rc\rc$.

\begin{Lemma}
\label{lem:jointconvbootstrapUstat}
   Under Assumption \ref{Ass:bootstrap} we have $\lc nU_n^\star h_{\theta_0,n},\sqrt{n}\lc U_n^\star \nabla h_{\theta_0}-U_n \nabla h_{\theta_0}\rc \rc\overset{\star}{\to}(Z_1,Z_2)$.
\end{Lemma}

Let $\tns$ denotes the estimator of $\too$ which is computed from the bootstrap sample, i.e., $\tns=\psi(X_1^\star.\ldots,X_n^\star)$. 
Similarly to the previous section, we need to assume that our bootstrapped estimator satisfies a bootstrap CLT jointly with $\lc nU_n^\star h_{\theta_0,n},\sqrt{n}\lc U_n^\star \nabla h_{\theta_0}-U_n \nabla h_{\theta_0}\rc \rc$. 
\begin{assumption}
\label{ass:joinconvboot}
    We assume that 
    $$\lc n U_n^\star h_{\too,n},\sqrt{n} \lc U_n^\star  \nabla_\theta h_{\too}- U_n \nabla_\theta h_{\too}\rc ,\sqrt{n}(\tns-\tn) \rc \overset{\star}{\to} (Z_1, Z_2,Z_3).$$
\end{assumption}
This is a high-level assumption which should be valid whenever $\thetanhat$ is based on an i.i.d.\ expansion. For example, if $\thetanhat$ is a non-degenerate $U$-statistic, \cite{arconesgine1992} provides the bootstrap CLT $\sqrt{n}(\tns-\tn)\overset{\star}{\to} Z_3$ and \cite[Remark 2.10 ii]{arconesgine1992} implies the required joint convergence.

A consequence of Lemma \ref{lem:jointconvbootstrapUstat} is that the naive extension of the bootstrap for degenerate $U$-statistics fails whenever $Z_2,Z_3\neq 0$.
\begin{prop}
\label{prop:naivebootfails}
   Under Assumptions \ref{Ass:bootstrap} and \ref{ass:joinconvboot} we have that $nU_n^\star h_{\thetanhat^\star,n}\overset{\star}{\to}  Z_1$.
\end{prop}

The proposition illustrates the we cannot naively apply the bootstrap for degenerate $U$-statistics when the parameter estimation influences its limiting law, since the limiting law of the naively bootstrapped degenerate $U$-statistic is equal to the limiting law without parameter estimation. To solve this issue, we suggest using a correction term that increases the variability of the bootstrapped degenerate $U$-statistic in just the right way to obtain the correct limiting law. 

\begin{Theorem}
\label{thm:bootstrapclt}
Under Assumptions \ref{Ass:bootstrap} and \ref{ass:joinconvboot}
\begin{align*}
    &n U_n^\star h_{\tns,n}+ n(\tns-\tn) ^T\lc U_n^\star  \nabla_\theta h_{\tns}- U_n  \nabla_\theta h_{\tn} \rc \\
    &\overset{\star}{\to}  Z_1+Z_2^T Z_3+Z_3^T \EE_{X,X^\prime\sim Q}\lk  \nabla^2_\theta h_{\too}(X,X^\prime) \rk Z_3 
 \end{align*}
 \end{Theorem}
To the best of our knowledge, the theorem provides the first general valid bootstrap scheme for degenerate $U$-statistics in the presence of parameter estimation. The increased generality however does come at the price of requiring access to $\nabla h_\theta$, which might not always be the case.

\begin{Remark}
\label{rem:bootstrapunderalt}
    Note that $\EE\lk h_{\theta_0}\rk=0$ must not be satisfied for Theorem \ref{thm:bootstrapclt} to hold. Therefore, under our assumptions for Theorem \ref{thm:bootstrapclt}, we always obtain a finite limit for the bootstrap of a parameter-dependent degenerate $U$-statistic, even when its mean is non-zero. This is particularly useful in a hypothesis testing framework, where often $H_0$ is of the form $\EE\lk h_\thetanull\rk=0$ and one usually needs to ensure that the bootstrapped test statistic is $O^\star(1)$ under $H_1^C$.
\end{Remark}

\subsection{Bootstrapping the KSD}
It remains to show that the framework for bootstrapping parameter-dependent degenerate $U$-statistics can be applied to the KSD. Denote
$\widetilde{\ksd}^2_q(p_{\tns}):= U_n^\star h_{\tns,n}+ (\tns-\tn) \lc U_n^\star  \nabla_\theta h_{\tns}- U_n  \nabla_\theta h_{\tn} \rc$, where $h_\theta$ denotes the core of the KSD test statistic, and recall that $n\hksd^2_q(p_\thetanhat)\to Z$ as specified in Theorem \ref{thm.: Convergence of KSD under H0}. Note that Lemma \ref{lemma:momentsofh} in the appendix verifies all moment conditions that are required to apply Theorem \ref{thm:bootstrapclt} for the KSD. Therefore, we have the following corollary:

 \begin{cor}
\label{cor:bootstrapCLTKSD}
 Under Assumptions \ref{KSD assumpt. 0}-\ref{ass:joinconvboot} and under $H_0^C$, we have that $n\widetilde{\ksd}^2_q(p_{\tns})\overset{\star}{\to}Z$. Moreover, under $H_1^C$, $n\widetilde{\ksd}^2_q(p_{\tns})=O^\star(1)$.
 \end{cor}

The result can be used to formulate a goodness-of-fit test based on the KSD.

\begin{algorithm}
\caption{A $\ksd$ goodness-of-fit test}
\label{algmodelspecification}
\begin{algorithmic}[1]
\Require I.i.d.\ sample $\lc X_i\rc_{1\leq i\leq n}$ from $Q$; estimator $\thetanhat$ of $\argmin_{\theta\in\Theta} \ksd_q(p_\theta)$; number $B$ of bootstrap replications; confidence level $\gamma$.
\State Compute $\thetanhat$ and $\hksd^2_q(p_{\thetanhat})$.
\For{$b = 1$ \textbf{to} $B$}
  \State Draw a sample $(X_i^\star)_{1\leq i\leq n}$ of size $n$ from $\lc X_i\rc_{1\leq i\leq n}$.
  \State Compute $\tns\!\big((X_i^\star)_{1\leq i\leq n}\big)$ and set $T^{(b)} \gets \widetilde{\ksd}^2_q\!\big(p_{\tns}\big)$.
\EndFor
\Ensure Reject $H_0^C$ when $\hksd^2_q(p_{\thetanhat}) > \mathrm{Quantile}\lc 1-\gamma;\, T^{(b)}_{1\leq b\leq B}\rc$; otherwise accept $H_0^C$.
\end{algorithmic}
\end{algorithm}

   \begin{Remark}[Wild bootstrap for the KSD]
       \cite{Key_2023} proposed a wild bootstrap to mimic the asymptotic distribution of the $V$-statistic $\hksd_q^2(p_\thetanhat)+((n-1)n)^{-1}\sum_{i=1}^n h_{\hat{\theta}_n}(X_i,X_i)$, without providing theoretical guarantees. It turns out that, by similar arguments as above, one can show that the asymptotic distribution of their bootstrap procedure is solely $Z_1+\e\lk h_{\theta_0} (X,X)\rk$, see Appendix \ref{app:inconswildboot} for the proof of this statement. Therefore, their bootstrap procedure does not yield a theoretically valid testing procedure as it ignores the influence of parameter estimation on the asymptotic distribution.
   \end{Remark}

\section{Simulations}

The simulation study analyses the finite sample performance of our composite KSD test. Since the theoretical results from the previous sections are asymptotic, a Monte Carlo study
is useful to verify how well the approximations hold for moderate sample sizes. 
We focus on the empirical
level and power of the test described in Algorithm \ref{algmodelspecification}. 
For the kernel $k$, we choose the Gaussian kernel
\[
k(x,y) \;=\; \exp\!\Big(-\tfrac{1}{2\ell^2}\,\|x-y\|^2\Big), \qquad x,y\in\R^d,
\]
since it is the standard choice in the literature and has appealing properties, such as being integrally strictly positive definite and in the Stein class of every continuously differentiable density with support $\R^d$. The choice of the bandwidth $l$ is up to the user and a common default choice is to use the so-called median heuristic $2l^2=\text{Median}\lc\lc\Vert X_i -X_j\Vert^2\rc_{i,j\in [n],i\neq j}\rc$. However, the median heuristic makes $l$ a data-dependent estimator. This could potentially alter the limiting law of $n\hksd^2_q(p_{\thetanhat})$, similarly to what we have seen in the previous sections. Therefore, we cannot simply apply the median heuristic here and instead have to use a deterministic bandwidth. One can show that in all our experiments $\text{Median}\lc\lc\Vert X_i -X_j\Vert^2\rc_{i,j\in [n],i\neq j}\rc=O(d)$, which implies that a reasonable choice for $l$ is $c\sqrt{d}$, where $c$ is a fixed tuning constant. To determine the tuning constant we mimicked the procedure of \cite{Key_2023} to chose $c$ from the grid $\{0.15,0.2,0.3,0.5,1.0\}$ by minimizing the deviation from the level, resp.\ maximizing the power of the test, on a separate artificial dataset.

\subsection{Monte Carlo study under the null hypothesis}
To the best of our knowledge, no existing KSD-based test is capable of handling composite hypotheses while asymptotically controlling the test level under the null. Nevertheless, to assess the finite sample performance, we compare our test with the wild bootstrap-based KSD test of \cite{Key_2023}, which is shown to not keep its level in Appendix \ref{app:inconswildboot}.

To assess the performance of the two tests under the null, let $X\in\R^d$ denote a $d$-dimensional multivariate standard normal random vector and set
\[
Q \;=\; \mathcal N(0,I_d), 
\qquad X_1,\dots,X_n \stackrel{\text{i.i.d.}}{\sim} Q.
\]
The family of models $\lc P_\theta\rc_{\theta\in\Theta}$ is the $d$-variate Gaussian family with unknown mean and covariance and the optimal parameter is given by $\theta_0=(0,I_d)$, which is estimated via the empirical mean and covariance matrix. We conducted $500$ independent Monte Carlo replications and set the number of bootstrap replications for both our bootstrap and the wild bootstrap to $200$. Figure \ref{fig:nullhyp} shows the empirical rejection probabilities for an asymptotic level of $5\%$ for dimensions $d\in\{1,4\}$ and sample sizes $n\in\{200,300,400,500,600\}$, where additional illustrations can be found in Figure \ref{fig:nullhypapp}.

\begin{figure}[h]
        \centering
    \begin{subfigure}{0.49\linewidth}
        \centering
         \includegraphics[width=\linewidth]{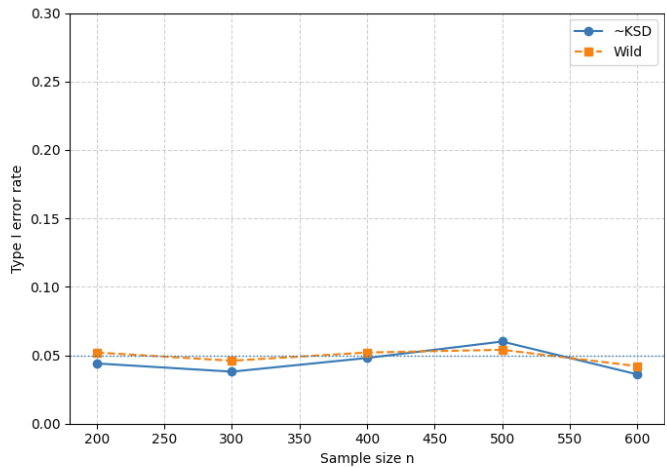} 
    \end{subfigure}
    \hfill
    \begin{subfigure}{0.5\linewidth}  
    \centering
    \includegraphics[width=\linewidth]{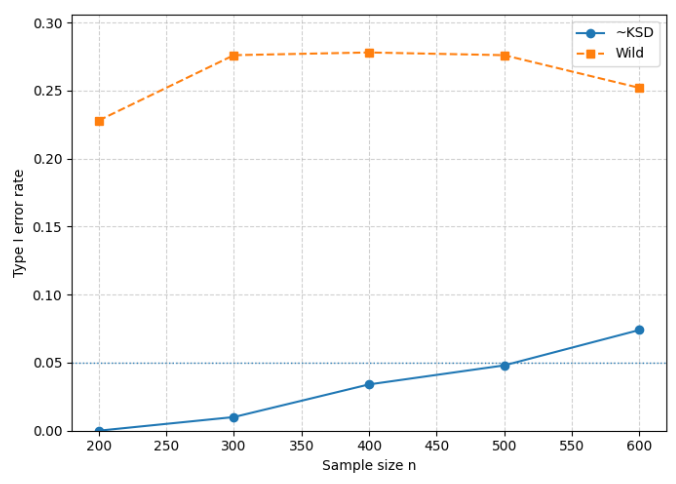}
    \end{subfigure}
    \caption{Simulation under the null hypothesis for dimension $d=1$ (left) and dimension $d=4$ (right) with $c=0.2$. KSD denotes the test proposed in this paper whereas Wild denotes the KSD test from \cite{Key_2023}.}
    \label{fig:nullhyp}
\end{figure}

One can see that in dimension $d=1$, both tests keep their level equally well. However, for dimension $d=4$, we can see that the wild bootstrap based test from \cite{Key_2023} does not keep its level, whereas our proposed test keeps its level reasonably well, in line with the theoretical findings from the previous sections.

\subsection{Monte Carlo study under the alternative hypothesis}

Under the alternative hypothesis, we again benchmark our test against the wild bootstrap-based KSD test from \cite{Key_2023} as well as the MMD-based composite tests proposed in \cite{brueckminfermanian2024} and 
\cite{Key_2023}. Our framework to assess the test under the alternative is as follows: We draw $n$ samples from a symmetric two-component Gaussian mixture,
\[
Q_\mu \;=\; \tfrac12\,\mathcal N(e_1\mu, I_d) \,+\, \tfrac12\,\mathcal N(-e_1\mu,I_d),
\]
with mixing weights $1/2$ and separation parameter $\mu\ge0$, where $e_1$ denotes the first unit vector $(1,0,\ldots, 0)\in\R^d$.
Note that $Q_\mu\not\in \lc P_\theta\rc_{\theta\in \Theta}$ for $\mu>0$, hence $H_1:\;Q_\mu\neq P_\theta\ \forall\theta\in\Theta$ holds.
We conducted $300$ independent Monte Carlo replications and set the number of bootstrap replications $B$ to $200$.
Figure \ref{fig:alternative} reports the power of the tests for varying dimensions $d\in\{1,2,4\}$, sample sizes $n\in\{100,200,300,400,500\}$ and $\mu\in\{1,2\}$. Additional illustrations can be found in Figure \ref{fig:alternapp}.

\begin{figure}[h]
    \begin{subfigure}{0.43\linewidth}
            \hspace{-0.5cm}
        \centering
    \includegraphics[width=\linewidth]{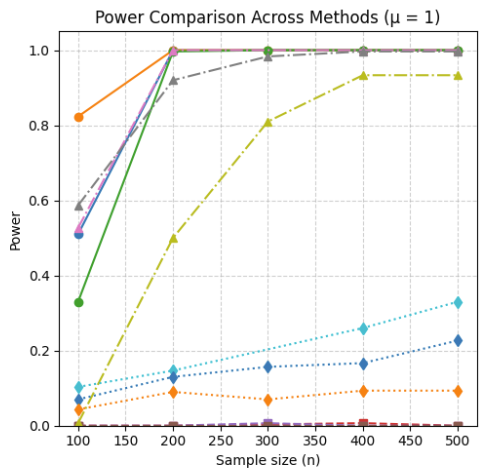} 
    \end{subfigure}
    \hspace{-0.35cm}
    \begin{subfigure}{0.49\linewidth}  
    \centering
    \includegraphics[width=1.20\linewidth]{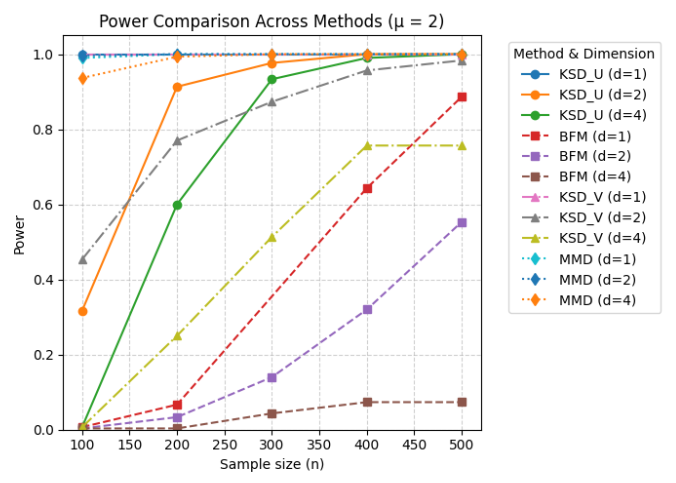}
    \end{subfigure}
    \caption{Empirical rejection probabilities under the alternative for dimensions $d\in\{1,2,4\}$, $\mu=1$ (left) and $\mu=2$ (right) with $c=1$. KSD$_U$ denotes the test proposed in this paper, KSD$_V$ (resp.\ MMD) denotes the wild bootstrapped KSD (resp.\ MMD) tests from \cite{Key_2023} and BFM denotes the MMD test from \cite{brueckminfermanian2024}}.
    \label{fig:alternative}
\end{figure}

Figure \ref{fig:alternative} shows that all MMD based tests have a significantly less power than the KSD based tests, independently of the dimension and separation parameter. Furthermore, the KSD based test from this paper outperforms the KSD test of \cite{Key_2023} for this example. Moreover, the power of the tests decreases with increasing dimension, which is due to the fact that a difference in only one component of a $d$-dimensional random vector is harder to uncover with increasing dimension. 

Altogether, the findings of the simulation study suggest that the KSD-based test proposed in this paper may have superior finite sample performance over the KSD-based test of \cite{Key_2023}. Moreover, it seems that both KSD-based tests have superior power compared to the MMD-based tests of \cite{brueckminfermanian2024} and \cite{Key_2023}. This is not surprising, as the MMD-based tests do not incorporate any information about the underlying density of the fitted family of distributions, in contrast to the KSD. Generalizing these simulation results to other scenarios would require more comprehensive analysis, which, however, lies beyond the scope of this paper.

\bibliography{bibliography.bib}

\section*{Acknowledgments}
The authors would like to express their deepest gratitude to Aleksey Min. Large parts of this project were conducted under his master thesis supervision of Veronika Reimoser at TUM. Moreover, we would like to thank the whole Chair of Mathematical Finance at TUM for allowing us to use their computational resources.
 
\section*{Funding}
This work was supported by the Swiss National Science Foundation under Grant 186858.

\appendix
\section{Proofs}

\subsection{Technical results}

\begin{Lemma}
\label{lemma:momentsofh}
  Assumption \ref{KSD assumpt. 5} and \ref{KSD assumpt. 3} imply 
  \begin{enumerate}
    \item[(i)] $\e\lk \sup_{\theta\in\Theta}\Vert h_\theta(X,X')  \Vert_2^2 \rk<\infty$ and $\e\lk \sup_{\theta\in\Theta}\Vert h_\theta(X,X)  \Vert_2^2 \rk<\infty$,
    \item[(ii)] $\e\lk \sup_{\theta\in\Theta}\Vert \nabla_\theta h_\theta(X,X')  \Vert_2^2 \rk<\infty$ and $\e\lk \sup_{\theta\in\Theta}\Vert \nabla_\theta h_\theta(X,X)  \Vert_2^2 \rk<\infty$,
    \item[(iii)] $\e\lk \sup_{\theta\in\Theta}\Vert \nabla^m_\theta h_\theta(X,X')  \Vert_1 \rk<\infty$ and $\e\lk \sup_{\theta\in\Theta}\Vert \nabla^m_\theta h_\theta(X,X)  \Vert_1 \rk<\infty$ for all $2\leq m\leq 3$.
  \end{enumerate} 
Further, for all $\theta\in\Theta$ we have $\e\lk  \nabla^m_\theta h_{\theta}(X,X')\rk=\nabla^m_\theta\e\lk   h_{\theta}(X,X')\rk$ for all $1\leq m\leq 3$.
\end{Lemma}

\begin{proof}
    Note that the last statement of the lemma follows from $(i)-(iii)$ by an application of the Leibniz rule with the majorant $\sup_{\theta\in\Theta}\nabla^m_\theta h_\theta(X,X')$. Thus, it remains to show $(i)-(iii)$. Since $k$ and its partial derivatives are bounded we can forget about their influence when showing finiteness of expectations.

    From Assumption \ref{KSD assumpt. 3}, we can deduce that
    \begin{align*}
        &\e\lk \sup_{\theta\in\Theta} \lv s_{p_\theta,i}(X)\rv^a \sup_{\theta\in\Theta} \lv s_{p_\theta,j}(X)\rv^v  \rk \\
        &\leq \e\lk \sup_{\theta\in\Theta} \lv s_{p_\theta,i}(X)\rv^{2a} \rk^{1/2} \e\lk \sup_{\theta\in\Theta} \lv s_{p_\theta,j}(X)\rv^{2b}  \rk^{1/2}<\infty
    \end{align*} 
    for all $a,b\in\{0,1,2\}$, and similarly, for $\e\lk \sup_{\theta\in\Theta}  \lv s_{p_\theta,j}(X)\rv \sup_{\theta\in\Theta} \lv s_{p_\theta,i}(X)\rv  \rk $.  Therefore, since $\e\lk h(X,X')\rk$ and $\e\lk h(X,X)\rk$ can be estimated by such terms, $(i)$ is satisfied.

    Next we show $(ii)$. First, observe that $ \partial_{\theta_j} h_\theta(X,X) $ is given by 
    \begin{align}
        &k(X,X) \sum_{1\leq i\leq p} 2  s_{p_\theta,i}(X) \partial_{\theta_j} s_{p_\theta,i}(X)+ \nabla_{x} k(X,X)^T\partial_{\theta_j} s_{p_\theta}(X) \nonumber\\
        &+\nabla_{x'} k(X,X)^T\partial_{\theta_j} s_{p_\theta}(X). \label{eqnparder}
    \end{align}

    Now, let us show $(ii)$. When considering $\Vert \nabla_\theta h_\theta(X,X)  \Vert_2^2$, we only need to show the finiteness of expectations of squares of terms appearing in (\ref{eqnparder}). For this purpose it is enough to observe that
    \begin{align*}
         &\e\lk \sup_{\theta\in\Theta} \lv s_{p_\theta,i}(X)^{a} \partial_{\theta_j} s_{p_\theta,i}(X)^{b}   \rv\rk\\
         &\leq \e\lk \sup_{\theta\in\Theta} \lv s_{p_\theta,i}(X)^{2a}   \rv\rk^{1/2} \e\lk \sup_{\theta\in\Theta} \lv \partial_{\theta_j} s_{p_\theta,i}(X)^{2b}   \rv\rk^{/1/2} <\infty
    \end{align*}
    for $a,b\in\{0,1,2\}$.
    Thus, $(ii)$ follows, as the argument for $\Vert \nabla_\theta h_\theta(X,X')  \Vert_2^2$ is similar, but fewer moments are needed due to the independence of $X,X'$.

    It remains to show $(iii)$. Observe that each element in $ \nabla^m_\theta h_\theta(X,X') $ is of the form 
    \begin{align}
       &k(X,X^\prime) \sum_{1\leq i\leq p}\sum_{0\leq j\leq m}\partial^{j}_\theta s_{p_\theta,i}(X)\partial^{m-j}_\theta s_{p_\theta,i}(X') \nonumber\\
       &+ \nabla_{x} k(X,X')^T \partial^{m}_\theta s_{p_\theta}(X)+\nabla_{x^\prime} k(X,X')^T \partial^{m}_\theta s_{p_\theta}(X') \label{eqnpardermix} 
    \end{align}
    and each element in $ \nabla^m_\theta h_\theta(X,X) $ of the form 
    \begin{align*}
       &k(X,X) \sum_{1\leq i\leq p}  \sum_{1\leq j\leq m} 2\partial_\theta^{m-j} s_{p_\theta,i}(X) \partial^j_{\theta} s_{p_\theta,i}(X)\\
       &+ \nabla_{x} k(X,X)^T\partial^m_{\theta} s_{p_\theta}(X)+\nabla_{x'} k(X,X)^T\partial^m_{\theta} s_{p_\theta}(X),
    \end{align*}
    where $0\leq j\leq m$, $\partial_\theta^j$ is an abstract notation for $\frac{\partial^j}{\partial \theta_{i_1}\ldots\partial \theta_{i_j}}$ for some $(i_l)_{1\leq l\leq j}\in[p]$, and with the convention $\partial^0 f(\theta)=f(\theta)$.    
    Therefore, we get that $\e\lk \sup_{\theta\in\Theta}\Vert \nabla^m_\theta h_\theta(X,X')  \Vert_1 \rk<\infty$ and  $\e\lk \sup_{\theta\in\Theta}\Vert \nabla^m_\theta h_\theta(X,X)  \Vert_1 \rk<\infty$ whenever
    \begin{align*}
     &\e\lk \sup_{\theta\in\Theta}\lv   \partial^{j}_\theta s_{p_\theta,i}(X)\partial^{m-j}_\theta s_{p_{\theta},i}(X') \rv\rk \\
     &\leq \e\lk \sup_{\theta\in\Theta}\lv\partial^{j}_\theta s_{p_{\theta},i}(X) \rv \rk
      \e\lk \sup_{\theta\in\Theta}\lv\partial^{m-j}_\theta s_{p_{\theta},i}(X') \rv \rk<\infty
    \end{align*}
    and
    \begin{align*}
     &\e\lk \sup_{\theta\in\Theta}\lv   \partial^{j}_\theta s_{p_{\theta},i}(X)\partial^{m-j}_\theta s_{p_{\theta},i}(X) \rv\rk \\
     &\leq \e\lk \sup_{\theta\in\Theta}\lv\partial^{j}_\theta s_{p_{\theta},i}(X) \rv^2 \rk^{1/2} 
      \e\lk \sup_{\theta\in\Theta}\lv\partial^{m-j}_\theta s_{p_{\theta},i}(X') \rv^2 \rk^{1/2} <\infty,
    \end{align*}
    which is satisfied by Assumption \ref{KSD assumpt. 3}.
    Thus, $\e\lk \sup_{\theta\in\Theta}\Vert \nabla^m_\theta h_\theta(X,X')  \Vert_1 \rk<\infty$. 
\end{proof}

Let $H:=(H_{h,l})_{1\leq h,l\leq p}$ be defined via the maps $H_{h,l}:\Theta \times \XX \times \XX\to \RR$ as
\begin{equation}\label{def.: H map definition}
H_{h,l}(\theta, x,x')
= \frac{\partial^2}{\partial \theta_h \partial\theta_l} 
h_\theta (x,x').
\end{equation}

\begin{cor}\label{cor:asymptist}
Assume that $(X_i)_{i\in\NN}\overset{i.i.d.}{\sim } Q$. Under $H_0^C$ and the conditions of Lemma \ref{lemma:momentsofh} we have
$n \hksd^2_q(p_\thetanull)\xrightarrow{d}Z_1$ and $\sqrt{n}\frac{1}{n(n-1)} \sum_{i,j\in[n], i\neq j}\nabla_{\theta} h_\thetanull(X_i,X_j) \xrightarrow{d} Z_2$, where $Z_1\sim \sum_{j=1}^{\infty} \lambda_j (T_j^2 - 1)$ with $(T_j)_{j\in\NN}\overset{i.i.d.}{\sim}\mathcal{N}(0,1)$ and the $\lambda_j$ are the eigenvalues of the operator $ \psi\mapsto \EE_{X\sim Q}[h_\thetanull(x,X) \psi(X)]$ as well as $Z_2\sim \mathcal{N}(0,4\Var\lc \e_X\lk \nabla_{\theta} h_{\theta_0}(X,X^\prime) \rk\rc$. Moreover $\frac{1}{n(n-1)}\sum_{i,j\in[n], i\neq j}H_{h,l}(\thetanull, X_i,X_j)\to H^\star_{h,l} $ for all $h,l\in[p]$.
\end{cor}
\begin{proof}
    Follows from standard $U$-statistics theory since the respective moment conditions are satisfied by Lemma \ref{lemma:momentsofh}.
\end{proof}

\subsection{Proof of Theorem \ref{thm.: Convergence of KSD under H0}}

By Assumption \ref{KSD assumpt. 3}, for all $x\in\XX$ we have that $p_\theta(X)>0$ almost surely for $X\sim Q$ and that $p_\theta(x)\in\mathcal{C}^4(\theta)$, so $h_\theta(x,x') \in \mathcal{C}^3(\theta)$ for any $x,x'\in\XX$. We are therefore able to perform a second order Taylor expansion around $\theta_0$, that yields
\begin{align*}
&n \hksd^2_q(p_\thetanhat) \\
&= n \hksd^2_q(p_\thetanull) + \sqrt{n} (\thetanhat - \theta_0)^T \cdot \sqrt{n} \  \frac{1}{n(n-1)} \sum_{i,j\in[n], i\neq j}\nabla_{\theta} h_\thetanull(X_i,X_j)  \\
&+ \sqrt{n} (\thetanhat - \theta_0)^T
\cdot \bigg( \frac{1}{n(n-1)} \sum_{i,j\in[n], i\neq j}H(\thetanull, X_i,X_j) \bigg)
\cdot \sqrt{n} (\thetanhat - \theta_0)  + R(\thetanhat), 
\end{align*}
where $R(\thetanhat)$ denotes the random remainder term and $H$ is defined in (\ref{def.: H map definition}). Note that by Lemma \ref{lemma:momentsofh} we have $\e\lk \nabla_{\theta} h_\thetanull(X,X')\rk=\nabla_{\theta}\e\lk  h_\thetanull(X,X')\rk=\nabla_{\theta}\ksd_q(p_\thetanull)=0$ as $\theta_0$ is the unique minimizer of $\theta\mapsto \ksd^2_q(p_{\theta})$ by Assumption \ref{KSD assumpt. 3}. Moreover, $\e\lk H(\thetanull, X_i,X_j)\rk=\nabla_\theta^2\e\lk h_\thetanull( X,X')\rk=H^*$. Thus, the result immediately follows from Assumption \ref{KSD assumpt. 6} if we can show  that $R(\thetanhat)=o_{\PP}(1)$.
By Taylor's theorem, the remainder term is of the form
\begin{align*}
R(\thetanhat)&=n\sum_{i,j,h\in[p]} R_{i,j,h}(\thetanhat) \prod_{l\in\{i,j,h\}}(\hat{\theta}_{n,l}-\theta_{o,l}) \quad \text{with } 
\lVert R_{i,j,h}(\thetanhat)\rVert \\
&\leq 
 \frac{1}{3!} \sup_{\theta\in\Theta} \left\lVert \nabla_\theta^3 \hksd^2_q(p_{\theta}) \right\rVert_\infty. 
 \end{align*}
For $n\in\NN$, we can upper bound 
$$\sup_{\theta\in\Theta} \left\lVert \nabla_\theta^3 \hksd_q^2(p_\theta) \right\rVert_\infty
\leq 
\frac{1}{n(n-1)}\sum_{i,j\in[n], i\neq j} 
\sup_{\theta\in\Theta} \left\lVert \nabla_\theta^3 h_\theta(X_i,X_j) \right\rVert_\infty.
$$
The r.h.s. of this is a $U$-statistic. Since $\EE_{X,X'\sim Q}\big[\sup_{\theta\in\Theta} \left\lVert \nabla_\theta^3 h_\theta(X,X') \right\rVert_1 \big]<\infty$, standard $U$-statistic results yield that the $U$-statistic converges in distribution to its expectation, which is finite.
Thus, $\sup_{\theta\in\Theta} \lVert \nabla_\theta^3 \hksd^2_q(p_\theta) \rVert_\infty$ is bounded in probability, yielding 
$$R(\thetanhat)=O_{\PP}(n\lVert \thetanhat-\thetanull\rVert^3)=O_{\PP}(\lVert \thetanhat-\thetanull\rVert)=o_{\PP}(1)
$$
and the result follows.

\subsection{Proof of Theorem \ref{thm.: KSD Consistency under H_1^C}}

First, recall that by our assumptions $\theta\mapsto \ksd_q(p_\theta)$ is continuous. Due to the compactness of $\Theta$ we get that $H_1^C$ implies $\ksd_q(p_\thetanull)>0$. 
We prove the theorem by contradiction and assume there exists an event $A$ with $\PP(A)>0$ such that $\liminf_{n\to\infty} \hksd^2_q(p_\thetanhat)=0$ on $A$.
This means that there exists a collection of indices $(a_n)_{n\geq 1}$ such that the subsequence $(\theta_{a_n})_{n\geq 1}$ satisfies
\begin{equation*}
\lim_{n \to \infty} \hksd_q^2(p_{\hat{\theta}_{a_n}})=0 \quad \text{on $A$.} 
\end{equation*}

Additionally, since $\Theta$ is compact, the Bolzano-Weierstrass theorem implies that the sequence $(\hat{\theta}_{a_n})_{n\in\NN}$ has a subsequence $(\hat{\theta}_{b_n})_{n\in\NN}$ that converges towards a $\theta^*\in\Theta$. By the mean value theorem,
\begin{align*}
  &\lv  \hksd_q^2(p_{\hat{\theta}_{b_n}})-\hksd_q^2(p_{\theta^*})\rv\\
  &\leq \Vert \hat{\theta}_{b_n}-\theta^*\Vert_1 \frac{1}{n(n-1)}\sum_{i,j\in[n],i\neq j}\sup_{\theta\in\Theta}\Vert \nabla_{\theta}h(X_i,X_j) \Vert_1.
\end{align*}
Since $\e\lk \sup_{\theta\in\Theta}\Vert \nabla_{\theta}h(X_i,X_j) \Vert_1 \rk<\infty$ by Lemma, \ref{lemma:momentsofh} we get $\lv  \hksd_q^2(p_{\hat{\theta}_{b_n}})-\hksd_q^2(p_{\theta^*})\rv\to 0$, implying that $\ksd_q(p_{\theta^\star})=\lim_{n\to\infty} \hksd_q^2(p_{\theta^*})=\lim_{n\to\infty} \hksd_q^2(p_{\hat{\theta}_{b_n}})=0$, a contradiction to $H_1^C$.

\subsection{Proof of Lemma \ref{lem:jointconvbootstrapUstat}}
 Lemma \ref{lemma:momentsofh} immediately implies that the conditions for an application of \cite[Theorem 2.4 and Corollary 2.6]{arconesgine1992} are satisfied. As a consequence, we obtain that $nU_n^\star h_{\theta_0,n}\to Z_1$ and $\sqrt{n}\lc U_n^\star \nabla h_{\theta_0}-U_n \nabla h_{\theta_0}\rc\to Z_2$ (see \cite[Remark 2.7]{arconesgine1992}). Joint convergence immediately follows from \cite[Remark 2.10 ii]{arconesgine1992}.

\subsection{Proof of  Proposition \ref{prop:naivebootfails}}
    Apply a third-order Taylor expansion to $U_n^\star h_{\thetanhat^\star,n}$ to get
 \begin{align*}
     U_n^\star h_{\tns,n}&=U_n^\star h_{\too,n}+(\tns-\too)^T U_n^\star \nabla_\theta h_{\too,n}\\
     &+(\tns-\too)^T\lc U_n^\star \nabla^2_\theta h_{\too,n} \rc(\tns-\too)+R_n(\tns).
 \end{align*}
 Note that by \cite[Theorem 2.4 a) and c)]{arconesgine1992}, $U_n^\star \nabla_\theta h_{\too,n}$ is the bootstrapped version of a $U$-statistic with degenerate core
\begin{align}
     &\nabla_\theta h_{\too} -\EE_{X\sim Q}\lk \nabla_\theta h_{\too}(X,\cdot)\rk-\EE_{X\sim Q}\lk \nabla_\theta h_{\too}(\cdot,X)\rk  \nonumber\\
     &+\EE_{X,X^\prime \sim Q}\lk \nabla_\theta h_{\too}(X,X^\prime)\rk \label{eqndegcore}
\end{align} 
and the same is true for $U_n^\star \nabla^2_\theta h_{\too,n}$.
Thus, they are $O^\star(n^{-1})$. Therefore, $n U_n^\star h_{\tns,n}=nU_n^\star h_{\too,n}+O^\star(n^{-1/2}) +nR_n(\tns)\overset{\star}{\to} Z_1$ if $R_n(\tns)=o^\star(n^{-1})$. To see that $R_n(\tns)=o^\star(n^{-1})$, observe that $R_n(\tns)$ is a sum of terms of the form
$$ \tau:=(\tns-\too)_i(\tns-\too)_j(\tns-\too)_k U^\star_n\lc \frac{\partial^3}{\partial\theta_i\partial\theta_j\partial\theta_k}h_{\Tilde{\theta},n}\rc, $$
where $\Tilde{\theta}$ is in the line segment joining $\tns$ and $\theta_0$. Since one can check that Assumption \ref{Ass:bootstrap} implies that
$$\EE\lk \lv \frac{\partial^3}{\partial\theta_i\partial\theta_j\partial\theta_k}h_{\Tilde{\theta},n}(X,X^\prime)\rv \rk\leq \EE\lk \sup_{\theta\in\Theta} \lv \frac{\partial^3}{\partial\theta_i\partial\theta_j\partial\theta_k}h_{\Tilde{\theta},n}(X,X^\prime)\rv\rk<\infty$$
and one similarly obtains $\EE\lk \lv \frac{\partial^3}{\partial\theta_i\partial\theta_j\partial\theta_k}h_{\Tilde{\theta},n}(X,X)\rv \rk<\infty$
the bootstrap law of large numbers \cite[Theorem 4.1]{arconesgine1992} implies $\tau=O^\star(n^{-3/2})$, proving the claim.

\subsection{Proof of  Theorem \ref{thm:bootstrapclt}}
 From the proof of Proposition  \ref{prop:naivebootfails}, we already know that $n U_n^\star h_{\tns,n}=n U_n^\star h_{\thetanull}+o^\star(1)$. Thus, it remains to investigate the remaining term
 \begin{align*}
    &U_n^\star  \nabla_\theta h_{\tns}- U_n \nabla_\theta h_{\tn} \\
    &= U_n^\star  \nabla_\theta h_{\too}- U_n \nabla_\theta h_{\too} +U_n^\star \nabla^2_\theta h_{\too}(\tns-\too)\\
    &-U_n \nabla^2_\theta h_{\too}(\tn-\too)+R_n(\tns,\tn)\\
      &= U_n^\star  \nabla_\theta h_{\too}- U_n \nabla_\theta h_{\too} +  U_n^\star \nabla^2_\theta h_{\too}   (\tns-\tn)\\
      &+  \lc U_n^\star \nabla^2_\theta h_{\too}-U_n \nabla^2_\theta h_{\too} \rc (\tn-\too) +R_n(\tns,\tn).
 \end{align*}
 Note that each term in $R_n(\tns,\tn)$ is of the form
 $$ (\tns-\too)_i (\tns-\too)_j \lc U_n^\star \frac{\partial^3}{\partial \theta_i\partial \theta_j \partial\theta_k} h_{\Tilde{\theta}_1}\rc $$
  or
  $$-(\tn-\too)_i (\tn-\too)_j U_n \frac{\partial^3}{\partial \theta_i\partial \theta_j \partial\theta_k} h_{\Tilde{\theta}_2},$$
 where $\Tilde{\theta}_1$ is on the line segment joining $\theta_0$ and $\tns$, and $\Tilde{\theta}_2$ is on the line segment joining $\theta_0$ and $\tn$. Therefore, this term is of order $O^\star(n^{-1})$, as all moment conditions required in \cite[Theorem 4.1]{arconesgine1992}  are satisfied for $\sup_{\theta\in\Theta}\nabla^3_\theta h_\theta$. 
 By similar arguments, $\lc U_n^\star \nabla^2_\theta h_{\too}-U_n \nabla^2_\theta h_{\too} \rc (\tn-\too)=O^\star(n^{-1})$.
 This implies that
 \begin{align*}
    &(\tns-\tn) ^T \lc U_n^\star  \nabla_\theta h_{\tns}- U_n  \nabla_\theta h_{\tn} \rc \\
    &=(\tns-\tn)^T\lc U_n^\star  \nabla_\theta h_{\too}- U_n \nabla_\theta h_{\too}  \rc\\
    &+(\tns-\tn)^T  U_n^\star \nabla^2_\theta h_{\too}(\tns-\tn) + O^\star(n^{-3/2}).
 \end{align*}
 Now, \cite[Theorem 4.1]{arconesgine1992} yields $ U_n^\star  \nabla^2_\theta h_{\too}\overset{\star}{\to}\EE_{X,X^\prime\sim Q}\lk  \nabla^2_\theta h_{\too}(X,X^\prime) \rk$ and we have derived in Lemma \ref{lem:jointconvbootstrapUstat} that $\sqrt{n}\lc U_n^\star  \nabla_\theta h_{\too}- U_n \nabla_\theta h_{\too}  \rc\to Z_2$. Therefore, by the joint convergence, Assumption \ref{ass:joinconvboot} and Slutsky's Lemma, we get 
 \begin{align*}
   &n U_n^\star h_{\tns,n}+ n(\tns-\tn)^T \lc U_n^\star  \nabla_\theta h_{\tns}- U_n  \nabla_\theta h_{\tn} \rc \\
   &=  n U_n^\star h_{\thetanull}+ (\tns-\tn)^T\lc U_n^\star  \nabla_\theta h_{\too}- U_n \nabla_\theta h_{\too}  \rc\\
   &+(\tns-\tn)^T \lc U_n^\star \nabla^2_\theta h_{\too} \rc (\tns-\tn) +o^\star(1)\\
   &\overset{\star}{\to} Z_1 +Z_2^T Z_3 +Z_3^T \EE_{X,X^\prime\sim Q}\lk  \nabla^2_\theta h_{\too}(X,X^\prime) \rk Z_3, 
 \end{align*} 
 which proves the claim.

 \subsection{Proof of Corollary \ref{cor:bootstrapCLTKSD}}
     Under $H_0^C$, it immediately follows from Theorem \ref{thm:bootstrapclt} that $n U_n^\star h_{\tns,n}+ n(\tns-\tn) \lc U_n^\star  \nabla_\theta h_{\tns}- U_n  \nabla_\theta h_{\tn} \rc \overset{\star}{\to}  Z$, noting that $\EE_{X,X^\prime\sim Q}\lk  \nabla^2_\theta h_{\too}(X,X^\prime) \rk=H^\star$. Under $H_1^C$, Remark \ref{rem:bootstrapunderalt} implies that $n\widetilde{\ksd}^2_q(\tns)\overset{\star}{\to} Y:=  Z_4 +Z_2^T Z_3+Z_3\e\lk \nabla^2_\theta h_{\theta_0}(X,X')\rk Z_3$, where $nU^\star_n h_{\theta_0,n}\overset{\star}{\to} Z_4$ with $Z_4\sim \lim_{n\to\infty} U_n \bar{h}_{\theta_0}$, where $\bar{h}_{\theta_0}$ is the degenerate $U$-statistic kernel obtained from centering $h_{\theta_0}$ as in (\ref{eqndegcore}). Thus, $Y=O^\star(1)$.

\section{Inconsistency of the naive wild bootstrap for $U$-statistics with estimated parameters}
\label{app:inconswildboot}
We implicitly assume in this section that all moment conditions for the convergence of the wild bootstrap are satisfied, which allows us to focus on the relevant issue of inconsistency of the wild bootstrap for degenerate $U$-statistics under parameter estimation. In \cite{Key_2023} the authors propose to bootstrap the $V$-statistic $n\hksd_q^2(p_\thetanhat)+n^{-1}\sum_{i\in[n]} h_{\tn}(X_i,X_i)$ via
$$ n^{-1}\sum_{i\neq j\in[n] }W_i W_j h_\thetanhat(X_i,X_j)+n^{-1}\sum_{i\in[n] }W_i W_i h_\thetanhat(X_i,X_i),$$
where $(W_i)_{i\in\NN}$ are i.i.d.\ Rademacher random variables independent of $(X_i)_{i\in\NN}$. It is easy to see that by the law of large numbers $n^{-1}\sum_{i\in[n] }W_i W_i h_\thetanhat(X_i,X_i)=n^{-1}\sum_{i\in[n] } h_\thetanhat(X_i,X_i)\to \e\lk h_{\theta_0}(X,X)\rk $. Moreover, $g_\thetanhat\lc (W,X),(W',X')\rc:=W W' h_\thetanhat(X,X')$ is a symmetric $U$-statistic core for the sample $\lc(W_i,X_i)\rc_{i\in\N}$. Additionally, for every symmetric function $f$ of two arguments, we have that $g_f((W,X),(W',X')):=W W'f(X,X')$ is a symmetric degenerate $U$-statistic core for the sample $\lc(W_i,X_i)\rc_{i\in\N}$, as
$$ \e_{W_1,X_1}\lk W_1 W_2f(X_1,X_2)\rk=0 $$
since $W_2$ is independent of $X_2$. Thus, we apply a first-order Taylor expansion to obtain
\begin{align*}
    &n^{-1}\sum_{i\neq j\in[n] }g_\thetanhat\lc (W,X),(W',X')\rc=n^{-1}\sum_{i\neq j\in[n] }g_\thetanull\lc (W,X),(W',X')\rc \\
    &+(\thetanhat-\thetanull)^T n^{-1}\sum_{i\neq j\in[n] }\nabla_\theta g_\thetanull\lc (W,X),(W',X')\rc      +o^\star(1).
\end{align*}
Now, by the previous arguments, we have
$$n^{-1}\sum_{i\neq j\in[n] }\nabla_\theta g_\thetanull\lc (W,X),(W',X')\rc=O_\PP(1)$$
 and
 $$ n^{-1}\sum_{i\neq j\in[n] }\nabla^2_\theta g_\thetanull\lc (W,X),(W',X')\rc=O_\PP(1).$$
Applying that $Z_n=O_\PP(1)\Rightarrow Z_n= O^\star(1)$, which can derived similarly as  \cite[Lemma 3]{cheng_huang}, we obtain that
\begin{align*}
    &n^{-1}\sum_{i\neq j\in[n] }g_\thetanhat\lc (W,X),(W',X')\rc\\
    &=n^{-1}\sum_{i\neq j\in[n] }g_\thetanull\lc (W,X),(W',X')\rc +o^\star(1)\overset{\star}{\to} Z_1, 
\end{align*} 
where $Z_1\sim\lim_{n\to\infty} n\hksd_q^2(p_{\thetanull})$ by the standard wild bootstrap for degenerate $U$-statistics. Therefore, the wild bootstrap procedure proposed by \cite{Key_2023} does not yield an asymptotically correct confidence interval as its limiting distribution is given by $Z_1+\e\lk h_{\theta_0}(X,X)\rk$, which disregards the influence of parameter estimation.

We remark that the calculations above in no means are dependent on the framework of the KSD. All claims hold for general parameter-dependent $U$-statistics, which shows that the naive wild bootstrap for degenerate $U$-statistics with estimated parameters does generally not provide the correct limiting law.

\section{Additional plots for the simulation study}
\label{app:addplots}
\begin{figure}[h]
    \centering
    \includegraphics[width=0.5\linewidth]{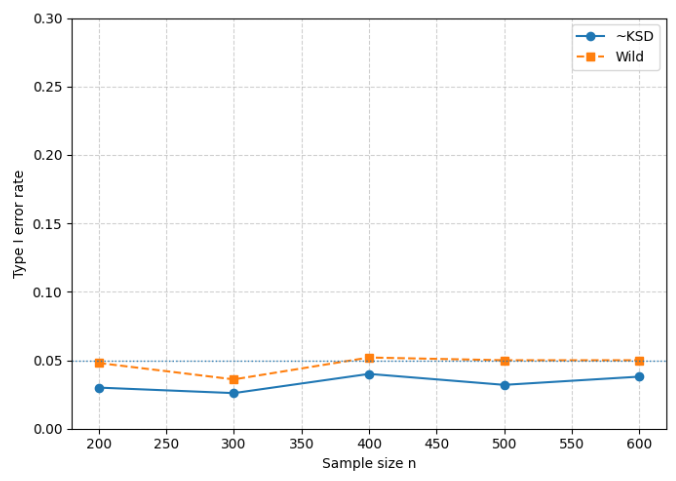}
    \caption{Simulation under the null hypothesis for dimension $d=2$. }
    \label{fig:nullhypapp}
\end{figure}
\begin{figure}[h]
    \centering
    \includegraphics[width=0.65\linewidth]{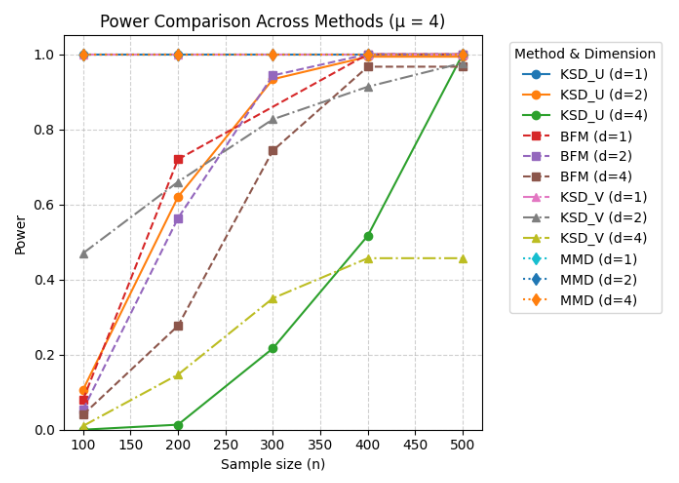}
        \caption{Simulation under the alternative for dimensions  $d\in\{1,2,4\}$ and $\mu=4$.}
        \label{fig:alternapp}
\end{figure}

\end{document}